\documentclass{article}
\usepackage{latexsym,amsfonts,amsmath,amsthm,amssymb,makeidx}
\usepackage[title]{appendix}
\usepackage{CJK,CJKnumb,CJKulem,times,dsfont,ifthen,mathrsfs,latexsym,amsfonts, color}
\usepackage{amsmath,amsthm,makeidx,fontenc,amssymb,bm,graphicx,psfrag,listings, curves,extarrows}
\usepackage{cite}

\let\oldbibliography\thebibliography
\renewcommand{\thebibliography}[1]{%
\oldbibliography{#1}%
\setlength{\itemsep}{0pt}%
}

\makeindex
\newtheorem{definition}{Definition}[section]
\newtheorem{theorem}{Theorem}[section]
\newtheorem{lemma}{Lemma}[section]
\newtheorem{corollary}{Corollary}[section]
\newtheorem{proposition}{Proposition}[section]
\newtheorem{remark}{Remark}[section]

\newcommand{\bt}{\begin{theorem}}
\newcommand{\et}{\end{theorem}}
\newcommand{\bl}{\begin{lemma}}
\newcommand{\el}{\end{lemma}}
\newcommand{\bd}{\begin{definition}}
\newcommand{\ed}{\end{definition}}
\newcommand{\bc}{\begin{corollary}}
\newcommand{\ec}{\end{corollary}}
\newcommand{\bp}{\begin{proof}}
\newcommand{\ep}{\end{proof}}
\newcommand{\bx}{\begin{example}}
\newcommand{\ex}{\end{example}}
\newcommand{\bi}{\begin{exercise}}
\newcommand{\ei}{\end{exercise}}
\newcommand{\bo}{\begin{prop}}
\newcommand{\eo}{\end{prop}}
\newcommand{\br}{\begin{remark}}
\newcommand{\er}{\end{remark}}
\newcommand{\be}{\begin{equation}}
\newcommand{\ee}{\end{equation}}
\newcommand{\ba}{\begin{align}}
\newcommand{\ea}{\end{align}}
\newcommand{\bn}{\begin{enumerate}}
\newcommand{\en}{\end{enumerate}}
\newcommand{\bg}{\begin{align*}}
\newcommand{\bcs}{\begin{cases}}
\newcommand{\ecs}{\end{cases}}

\newcommand{\sg}{\sigma}
\newcommand{\bean}{\begin{eqnarray*}}
\newcommand{\eean}{\end{eqnarray*}}

\numberwithin{equation}{section}

\begin{document}
\title{{\bf  On Isolated Singularities  of Fractional Semi-Linear Elliptic  Equations}\thanks {Supported by  NSFC.   \qquad E-mail addresses:  hui-yang15@mails.tsinghua.edu.cn (H.  Yang),   \qquad \qquad \qquad  wzou@math.tsinghua.edu.cn (W.  Zou)}
}

\date{}
\author{\\{\bf Hui  Yang$^{1}$,\;\;  Wenming Zou$^{2}$}\\
\footnotesize {\it  $^{1}$Yau Mathematical Sciences Center, Tsinghua University, Beijing 100084, China}\\
\footnotesize {\it  $^{2}$Department of Mathematical Sciences, Tsinghua University, Beijing 100084, China}
}

\maketitle
\begin{center}
\begin{minipage}{120mm}
\begin{center}{\bf Abstract}\end{center}

In this paper, we study the local behavior of nonnegative solutions of  fractional semi-linear equations 
$(-\Delta)^\sigma u = u^p$   with an isolated singularity, where $\sg \in (0,  1)$ and $\frac{n}{n-2\sg}  <  p  < \frac{n+2\sg}{n-2\sg}$.  We first use blow up method and a Liouville type theorem to derive an upper bound. Then we establish a monotonicity formula and a sufficient condition for removable singularity to  give a classification of the isolated singularities.  
When $\sg=1$,  this classification result has been proved by Gidas and Spruck (Comm. Pure Appl. Math. 34: 525-598, 1981).

\vskip0.10in


\end{minipage}

\end{center}

\vskip0.390in

\section{Introduction and Main results}
The purpose of this paper is to study the local behavior of nonnegative solutions of
\begin{equation}\label{Iso1}
(-\Delta)^\sg u= u^p~~~~~~~~~~~\textmd{in}~ B_1 \backslash  \{0\}
\end{equation}
with an isolated singularity at the origin, where the punctured unit ball $B_1 \backslash \{0\} \subset \mathbb{R}^n$ with $n \geq2 $, 
$\sg \in (0, 1)$ and $(-\Delta)^\sg$ is the fractional Laplacian.

When $\sg=1$, the isolated singularity of nonnegative solutions for \eqref{Iso1} has been very well understand, see Lions \cite{L} for $1 < p< \frac{n}{n-2}$,  Aviles \cite{A} for $p=\frac{n}{n-2}$, Gidas-Spruck \cite{G-S} for $\frac{n}{n-2}  <  p  < \frac{n+2}{n-2}$,  Caffarelli-Gidas-Spruck  \cite{C-G-S} for $\frac{n}{n-2}  \leq  p  \leq \frac{n+2}{n-2}$,  Korevaar-Mazzeo-Pacard-Schoen \cite{K-M-P-S} for $p=\frac{n+2}{n-2}$, and Bidaut-V\'{e}ron and V\'{e}ron \cite{B-V} for $p > \frac{n+2}{n-2}$.

The semi-linear equation \eqref{Iso1} involving the fractional  Laplacian has attracted a great deal of interest since they are of central importance in many fields,  such as see \cite{A-W,A-D-G-W,A-C-G-W,C-J-S-X,Ch-G,C-L-L,D-D-W,D-G,D-d-G-W,F-F,G-M-S,G-Q,J-d-S-X,J-L-X,K-M-W} and  references therein.  Recently,  the existence of singular solutions of equation \eqref{Iso1} with prescribed isolated singularities for  the critical exponent $p=\frac{n+2\sg}{n-2\sg}$ were studied in \cite{A-W,A-D-G-W,D-G,D-d-G-W} and for the subcritical regime $\frac{n}{n-2\sg}  <  p  < \frac{n+2\sg}{n-2\sg}$ were studied in \cite{A-W,A-C-G-W}.  Solutions  of \eqref{Iso1} with an isolated singularity are the simplest cases of those singular solutions.  In a recent paper \cite{C-J-S-X}, Caffarelli, Jin, Sire and Xiong study the local behavior of nonnegative solution of \eqref{Iso1} with $p=\frac{n+2\sg}{n-2\sg}$. More precisely, let $u$ be a nonnegative solution of \eqref{Iso1} with $p=\frac{n+2\sg}{n-2\sg}$ and suppose that the origin is not a removable singularity. Then,  near the origin
\begin{equation}\label{Iso2}
c_1 |x|^{-\frac{n-2\sg}{2}} \leq u(x) \leq c_2 |x|^{-\frac{n-2\sg}{2}},
\end{equation}
where $c_1, c_2$ are positive constants.

In this paper, we are interested in the local behavior of nonnegative solutions of \eqref{Iso1} with $\frac{n}{n-2\sg}  <  p  < \frac{n+2\sg}{n-2\sg}$.  For the classical case $\sg=1$,  this has been proved in the pioneering  paper \cite{G-S} by  Gidas and Spruck.

We study the equation \eqref{Iso1} via the well known extension theorem for the fractional Laplacian $(-\Delta)^\sg$ established by Caffarelli-Silvestre \cite{C-S}. We use capital letters, such as $X=(x, t)\in \mathbb{R}^n \times \mathbb{R}_+$,  to  denote  points in $\mathbb{R}_+^{n+1}$.  We also denote $\mathcal{B}_R$ as  the ball in $\mathbb{R}^{n+1}$ with radius $R$ and center at the origin,  $\mathcal{B}_R^+$ as the upper half-ball $\mathcal{B}_R\cap \mathbb{R}_+^{n+1}$,  and $\partial^0 \mathcal{B}_R^+$ as the flat part of $\partial \mathcal{B}_R^+$  which is the ball $B_R$ in $\mathbb{R}^{n}$.  Then the problem \eqref{Iso1} is equivalent to the following extension problem 
\begin{equation}\label{Iso3}
\begin{cases}
-\textmd{div}(t^{1-2\sg}  \nabla U)=0~~~~~~~~~& \textmd{in} ~ \mathcal{B}_1^+,\\
\frac{\partial U}{\partial \nu^\sg}(x, 0)=U^p(x, 0)~~~~~~~~& \textmd{on} ~ \partial^0\mathcal{B}_1^+ \backslash  \{0\},
\end{cases}
\end{equation}
where $\frac{\partial U}{\partial \nu^\sg}(x, 0) := -\lim_{t\rightarrow 0^+}t^{1-2\sg} \partial_t U(x, t)$.  By \cite{C-S}, we only need to analyze the behavior of the traces
$$
u(x):=U(x, 0)
$$
of the nonnegative solutions $U(x,t)$ of \eqref{Iso3} near the origin, from which we can get the behavior of solutions of  \eqref{Iso1} near the origin.

We say that $U$ is a nonnegative solution of \eqref{Iso3} if $U$ is in the weighted Sobolev space $W^{1, 2}(t^{1-2\sg}, \mathcal{B}^+_1 \backslash \overline{\mathcal{B}}^+_\epsilon)$ for every $\epsilon > 0$, $U\geq 0$, and it satisies \eqref{Iso3} in the sense of distribution away from 0, i.e., for every nonnegative $\Phi \in C_c^\infty\left((\mathcal{B}_1^+  \cup \partial^0\mathcal{B}_1^+) \backslash  \{0\}\right)$,
\begin{equation}\label{Solu}
\int_{\mathcal{B}_1^+} t^{1-2\sg} \nabla U \nabla \Phi=\int_{\partial^0\mathcal{B}_1^+} U^p \Phi.
\end{equation}
See \cite{J-L-X} for more details on this definition.  Then it follows from the regularity result in \cite{J-L-X} that $U$ is locally H\"{o}lder continuous in $\overline{\mathcal{B}}^+_1 \backslash \{0\}$.  We say that the origin 0 is a removable singularity of solution $U$ of \eqref{Iso3} if $U(x,0)$ can be extended as a continuous function near the origin, otherwise we say that the origin 0 is a non-removable singularity.  Our main result is the following
\begin{theorem}\label{THM01}
Let $U$ be a nonnegative solution of \eqref{Iso3}. Assume 
$$\frac{n}{n-2\sg} < p < \frac{n+2\sg}{n-2\sg}.$$
Then either  the singularity near 0 is removable, or there exist two positive constants $c_1$ and $c_2$ such that
\begin{equation}\label{Be01}
c_1 |x|^{-\frac{2\sg}{p-1}} \leq u(x) \leq c_2 |x|^{-\frac{2\sg}{p-1}}.
\end{equation}
\end{theorem}
\br\label{Re01}
We point out that, if \eqref{Be01} holds, then the Harnack inequality \eqref{P-3012} implies that
$$
C_1|X|^{-\frac{2\sg}{p-1}} \leq U(X) \leq C_2 |X|^{-\frac{2\sg}{p-1}}
$$
holds as well, for some positive constants $C_1$ and $C_2$.
\er
For the classical case $\sg=1$, Theorem \ref{THM01} were proved in \cite{G-S} by Gidas and Spruck.   We may also see \cite{C-G-S} for this  classical case. The similar upper bound in \eqref{Be01} obtained in \cite{G-S} for the classical case is very complicated and technical, here we use the blow up method and a Liouville type theorem  to prove  the upper bound in \eqref{Be01}.  To obtain the lower bound, there are some extra difficulties,  one of which is that the Pohozaev identity is not available.  More precisely, for the critical case $p=\frac{n+2\sg}{n-2\sg}$,  the Pohozaev integral $P(U, R)$ is independent of $R$ by the Pohozaev identity (see \cite{C-J-S-X} for more detalis on  $P(U,R)$). In \cite{C-J-S-X}, the authors make use of this property of the Pohozaev integral  to prove the lower bound, however, this does not hold in the subcritical case.  We will establish a useful monotonicity formula to overcome this difficulty. The others would be those extra techniques  to get the estimates of $U$ from those of its trace $u$.

The paper is organized as follow. In Section 2, we recall three propositions: a Liouville  theorem, a Harnack inequality and a Sobolev inequality. Section 3 is devoted to the proof  of Theorem \ref{THM01}. We first derive an upper bound and a special form of  Harnack inequality. Then we establish a monotonicity formula and a sufficient condition of removable singularity to prove Theorem \ref{THM01}.

\section{Preliminaries}
In this section, we introduce some notations and some propositions which will be used in our arguments. We denote $\mathcal{B}_R$ as the ball in $\mathbb{R}^{n+1}$ with radius $R$ and center 0, and $B_R$ as the ball in $\mathbb{R}^n$ with radius $R$ and center 0. We also denote $\mathcal{B}_R^+$ as the upper half-ball $\mathcal{B}_R\cap \mathbb{R}_+^{n+1}$,  $\partial^+\mathcal{B}_R^+=\partial \mathcal{B}_R^+ \cap \mathbb{R}_+^{n+1}$ as the positive part of   $\partial \mathcal{B}_R^+$,  and $\partial^0 \mathcal{B}_R^+=\partial \mathcal{B}_R^+ \backslash \partial^+\mathcal{B}_R^+$ as the flat part of $\partial \mathcal{B}_R^+$ which is the ball $B_R$ in $\mathbb{R}^n$.

We say $U \in W_{loc}^{1,2}(t^{1-2\sg}, \overline{\mathbb{R}^{n+1}_+})$ if $U \in W^{1,2}(t^{1-2\sg}, \mathcal{B}_R^+)$ for all $R>0$, and $U \in W_{loc}^{1,2}(t^{1-2\sg}, \overline{\mathbb{R}^{n+1}_+} \backslash \{0\})$ if $U \in W^{1,2}(t^{1-2\sg}, \mathcal{B}_R^+ \backslash \overline{\mathcal{B}}_\epsilon^+)$ for all $R > \epsilon >0$.

We next recall three propositions,  which will be used frequently in our paper. For convenience, we state them here. Their proofs can be found in  \cite{J-L-X}.  The first one is a Liouville type theorem.
\begin{proposition}\label{P1}
Let $U \in W_{loc}^{1,2}(t^{1-2\sg}, \overline{\mathbb{R}^{n+1}_+})$ be a nonnegative weak solution of
\begin{equation}\label{P11}
\begin{cases}
-\textmd{div}(t^{1-2\sg}  \nabla U)=0~~~~~~~~~& \textmd{in} ~\mathbb{R}^{n+1}_+,\\
\frac{\partial U}{\partial \nu^\sg}(x, 0)=U^p(x, 0)~~~~~~~~& \textmd{on} ~\mathbb{R}^n,
\end{cases}
\end{equation}
with
$$
1 \leq p < \frac{n + 2\sg}{n - 2\sg}.
$$
Then
$$
U(x, t) \equiv 0.
$$
\end{proposition}
The second one is a Harnack inequality, see also \cite{Cab-S}.
\begin{proposition}\label{P2}
Let $U \in W_{loc}^{1,2}(t^{1-2\sg}, \mathcal{B}^+_1)$ be a nonnegative weak solution of
\begin{equation}\label{P11}
\begin{cases}
-\textmd{div}(t^{1-2\sg}  \nabla U)=0~~~~~~~~~& \textmd{in} ~\mathcal{B}^+_1,\\
\frac{\partial U}{\partial \nu^\sg}(x, 0)=a(x)U(x, 0)~~~~~~~~& \textmd{on} ~\partial^0\mathcal{B}^+_1,
\end{cases}
\end{equation}
If $a\in L^q(B_1)$ for some $q > \frac{n}{2\sg}$, then we have
\be\label{P21}
\sup_{\mathcal{B}^+_{1/2}} U \leq C  \inf_{\mathcal{B}^+_{1/2}} U,
\ee
where $C$ depends only on $n, \sg$ and $\|a\|_{L^q(B_1)}$. 
\end{proposition}
The last one is a Sobolev type inequality.
\begin{proposition}\label{P3}
Let $D=\Omega \times (0, R) \subset \mathbb{R}^n \times \mathbb{R}_+$ with $R>0$ and $\partial \Omega$ Lipschitz.  Then there exists $C_{n, \sg}>0$ depending only on $n$ and $\sg$ such that
$$
\|U(\cdot, 0)\|_{L^{2n / (n-2\sg)}(\Omega)} \leq C_{n,\sg} \|\nabla U\|_{L^2(t^{1-2\sg}, D)}
$$ 
for all $U \in C_c^\infty (D \cup \partial^0 D)$.

\end{proposition}

\section{Classification of Isolated Singularities}
In this section, we investigate the local behavior of nonnegative solutions of \eqref{Iso3} and classify their isolated singularities. We first prove an upper bound and a special form of Harnack inequality for nonnegative solutions with a possible isolated singularity. We remark that this result will be of basic importance in classifying the isolated singularities.
\begin{proposition}\label{P-301}
Let $U$ be a nonnegative solution of \eqref{Iso3}, with $1 < p < \frac{n+2\sg}{n-2\sg}$. Then,
\begin{itemize}
\item [(1)]   there exists a positive constant $c$ independent of $U$ such that
\be\label{P-3011}
u(x) \leq c |x|^{-\frac{2\sg}{p-1}}~~~~~~~~~~ \textmd{in} ~ B_{1/2};
\ee
\item [(2)]  (Harnack inequality)
for all $0 < r < 1/8$, we have 
\be\label{P-3012}
\sup_{\mathcal{B}_{2r}^+ \backslash \overline{\mathcal{B}}_{r/2}^+} U \leq C \inf_{\mathcal{B}_{2r}^+ \backslash \overline{\mathcal{B}}_{r/2}^+} U,
\ee
where $C$ is a positive constant independent of $r$ and $U$.  In particular, for all $0 < r < 1/8$, we have 
\be\label{P-3013}
\sup_{\partial^+\mathcal{B}_{r}^+ } U \leq C \inf_{\partial^+\mathcal{B}_{r}^+ } U,
\ee
where $C$ is a positive constant independent of $r$ and $U$.
\end{itemize}
\end{proposition}
\begin{proof}
Suppose by contradiction that there exists a   sequence of points $\{x_k\} \subset B_{1/2}$ and a sequence of solutions $\{U_k\}$ of \eqref{Iso3},  such that  
\be\label{Upp01}
|x_k|^{\frac{2\sg}{p-1}} u_k(x_k)\rightarrow  +\infty~~~~~~\textmd{as} ~ k \rightarrow \infty.
\ee
As in \cite{C-J-S-X}, we define 
$$
v_k(x):=\left(\frac{|x_k|}{2} - |x-x_k|\right)^{\frac{2\sg}{p-1}} u_k(x),~~~~~|x - x_k| \leq \frac{|x_k|}{2}.
$$
Take $\bar{x}_k$ satisfy $|\bar{x}_k - x_k| < \frac{|x_k|}{2}$ and 
$$
v_k(\bar{x}_k) = \max_{|x-x_k|\leq\frac{|x_k|}{2}} v_k(x).
$$
Let 
$$
2\mu_k:=\frac{|x_k|}{2} - |\bar{x}_k-x_k|.
$$
Then
$$
0 < 2\mu_k \leq \frac{|x_k|}{2}~~~ \textmd{and}~~~\frac{|x_k|}{2} - |x-x_k|\geq \mu_k~~~~~\forall ~~  |x-\bar{x}_k| \leq \mu_k.
$$
By the definition of $v_k$, we have
$$
(2\mu_k)^{\frac{2\sg}{p-1}}u_k(\bar{x}_k) = v_k(\bar{x}_k) \geq v_k(x) \geq (\mu_k)^{\frac{2\sg}{p-1}} u_k(x)~~~~~\forall ~~  |x-\bar{x}_k| \leq \mu_k.
$$
Hence, we obtain
\be\label{Upp02}
2^{\frac{2\sg}{p-1}}u_k(\bar{x}_k) \geq u_k(x)~~~~~\forall ~~  |x-\bar{x}_k| \leq \mu_k. 
\ee
Moreover, by \eqref{Upp01}, we also have
\be\label{Upp03}
(2\mu_k)^{\frac{2\sg}{p-1}}u_k(\bar{x}_k) = v_k(\bar{x}_k) \geq v_k(x_k) =\left(\frac{|x_k|}{2}\right)^{\frac{2\sg}{p-1}} u_k(x_k)\rightarrow +\infty~~~~ \textmd{as} ~ k\rightarrow \infty.
\ee
Now, we define
$$
W_k(y, t):=\frac{1}{u_k(\bar{x}_k)} U_k\left(\bar{x}_k + \frac{y}{u_k(\bar{x}_k)^{\frac{p-1}{2\sg}}}, \frac{t}{u_k(\bar{x}_k)^{\frac{p-1}{2\sg}}}\right),
~~~~~ (y, t)\in \Omega_k,
$$
where
$$
\Omega_k:=\left\{(y, t)\in \mathbb{R}^{n+1}_+ |     \left(\bar{x}_k + \frac{y}{u_k(\bar{x}_k)^{\frac{p-1}{2\sg}}}, \frac{t}{u_k(\bar{x}_k)^{\frac{p-1}{2\sg}}} \right)    \in \mathcal{B}_1^+ \backslash \{0\}\right\}.
$$
Let $w_k(y):=W_k(y, 0)$. Then $W_k$ satisfies $w_k(0)=1$ and
\begin{equation}\label{Upp04}
\begin{cases}
-\textmd{div}(t^{1-2\sg}  \nabla W_k)=0~~~~~~~~~& \textmd{in} ~\Omega_k,\\
\frac{\partial W_k}{\partial \nu^\sg}(x, 0)=w_k^p~~~~~~~~& \textmd{on} ~\partial^0\Omega_k.
\end{cases}
\end{equation}
Furthermore, by \eqref{Upp02} and \eqref{Upp03}, we have
$$
w_k(y) \leq 2^{\frac{2\sg}{p-1}}~~~~~~~ \textmd{in} ~ B_{R_k}
$$
with
$$
R_k:=\mu_k u(\bar{x}_k)^{\frac{p-1}{2\sg}}\rightarrow +\infty~~~~ \textmd{as} ~ k \rightarrow \infty.
$$
By Proposition \ref{P2}, for any $T>0$, we have
$$
0 \leq W_k \leq C(T)~~~~~ \textmd{in} ~ B_{R_k/2} \times [0, T),
$$
where the constant $C(T)$ depends only on $n, \sg$ and $T$. By Corollary 2.10 and Theorem 2.15 in \cite{J-L-X} there exists $\alpha > 0$ such that for every $R>1$,
$$
\|W_k\|_{W^{1,2}(t^{1-2\sg}, \mathcal{B}_R^+)} + \|W_k\|_{C^\alpha(\overline{\mathcal{B}}_R^+)} + \|w_k\|_{C^{2, \alpha}(\overline{B}_R)} \leq C(R),
$$
where $C(R)$ is independent of $k$. Therefore,  there is a subsequence of $k\rightarrow \infty$, still denoted by itself,  and a nonnegative function $W\in W_{loc}^{1,2}(t^{1-2\sg}, \overline{\mathbb{R}^{n+1}_+}) \cap C_{loc}^\alpha(\overline{\mathbb{R}^{n+1}_+})$ such that, as $k\rightarrow \infty$,
$$
\begin{cases}
W_k \rightarrow W~~~~~~~\textmd{weakly} ~ \textmd{in} ~~~~~~~~~~&W_{loc}^{1,2}(t^{1-2\sg}, \overline{\mathbb{R}^{n+1}_+})  ,\\
W_k \rightarrow W~~~~~~~  \textmd{in}~~~~~~~&C_{loc}^{\alpha/2}(\overline{\mathbb{R}^{n+1}_+}),\\
w_k\rightarrow w~~~~~~~~~\textmd{in}~~~~~~~&C_{loc}^2(\mathbb{R}^n),
\end{cases}
$$
where $w(y)=W(y,0)$.  Moreover, $W$ satisfies $w(0)=1$ and 
\begin{equation}\label{Upp051}
\begin{cases}
-\textmd{div}(t^{1-2\sg}  \nabla W)=0~~~~~~~~~& \textmd{in} ~\mathbb{R}^{n+1}_+,\\
\frac{\partial W}{\partial \nu^\sg}(x, 0)=w^p~~~~~~~~& \textmd{on} ~\mathbb{R}^n.
\end{cases}
\end{equation}
This contradicts Proposition \ref{P1} and proves part (1) of the proposition.

Now we prove the Harnack inequality, which is actually a consequence of the upper bound \eqref{P-3011}. 
Let 
$$
V_r(X)=U(rX)
$$
for each $r\in (0, \frac{1}{8}]$ and for $\frac{1}{4} \leq |X| \leq 4$. Obviously, $V_r$ satisfies
\begin{equation}\label{Upp061}
\begin{cases}
-\textmd{div}(t^{1-2\sg}  \nabla V_r)=0~~~~~~~~~& \textmd{in} ~\mathcal{B}_4 \backslash \overline{\mathcal{B}}_{1/4},\\
\frac{\partial V_r}{\partial \nu^\sg}(x, 0)=a_r(x)v_r(x)~~~~~~~~& \textmd{on} ~B_4 \backslash \overline{B}_{1/4},\\
\end{cases}
\end{equation}
where $v_r(x)=V_r(x,0)$ and $a_r(x)=r^{2\sg} \left( u(rx) \right)^{p-1}$.  It follows \eqref{P-3011} that
$$
|a_r(x)| \leq C ~~~~~~~\textmd{for} ~ \text{all} ~1/4 \leq |x| \leq 4,
$$
where $C$ is a positive constant independent of $r$ and $U$.  By the Harnack inequality in Proposition \ref{P2} and the standard Harnack inequality for uniformly elliptic equations, we have
$$
\sup_{\frac{1}{2} \leq |X| \leq 2} V_r(X)  \leq C \inf_{\frac{1}{2}\leq |X| \leq 2} V_r(X),
$$
where $C$ is another positive constant independent of $r$ and $U$.  Hence, we get \eqref{P-3012}. 
\end{proof}

In order to prove the lower bound in \eqref{Be01}, we need to establish a monotonicity formula for the nonnegative solutions $U$  of \eqref{Iso3}.  More precisely, take  a nonnegative solutions  $U$  of \eqref{Iso3}, let $0 < r < 1$  and  define
$$
\aligned
E(r;U):= &  r^{2\frac{(p+1)\sg}{p-1} - n }\left[ r \int_{\partial^+\mathcal{B}_r^+} t^{1-2\sg} | \frac{\partial U}{\partial \nu} |^2 + \frac{2\sg}{p-1} \int_{\partial^+\mathcal{B}_r^+} t^{1-2\sg}\frac{\partial U}{\partial \nu} U\right] \\
& + \frac{1}{2}\frac{2\sg}{p-1}\left(\frac{4\sg}{p-1} - (n-2\sg)\right) r^{2\frac{(p+1)\sg}{p-1} - n -1}  \int_{\partial^+\mathcal{B}_r^+} t^{1-2\sg} U^2  \\
& -  r^{2\frac{(p+1)\sg}{p-1} - n+1 }  \left[ \frac{1}{2} \int_{\partial^+\mathcal{B}_r^+} t^{1-2\sg} |\nabla U|^2 -\frac{1}{p+1}\int_{\partial B_r} u^{p+1} \right].
\endaligned
$$
Then, we have the following monotonicity formula.  
\begin{proposition}\label{P302}
Let $U$ be a nonnegative solution of \eqref{Iso3} with $1 < p < \frac{n+2\sg}{n-2\sg}$. Then, $E(r; U)$ is non-decreasing in $r \in (0, 1)$. 
\end{proposition}
\bp
Take standard polar coordinates in $\mathbb{R}^{n+1}_+$:  $X=(x,t)=r\theta$, where $r=|X|$ and $\theta=\frac{X}{|X|}$.  Let $\theta_1=\frac{t}{|X|}$ denote the component of $\theta$ in the $t$ direction and 
$$\mathbb{S}^n_+=\{X\in \mathbb{R}^{n+1}_+ : r=1, \theta_1 >0\}$$
denote the upper unit half-sphere.
 
We use the classical change of variable in Fowler \cite{F},
$$
V(s, \theta)= r^{\frac{2\sg}{p-1}} U(r, \theta),~~~~~ s=\ln r.
$$
Direct  calculations show that $V$ satisfies
\begin{equation}\label{P30111}
\begin{cases}
V_{ss} - J_1 V_s - J_2V + \theta_1^{2\sg-1} \text{div}_\theta (\theta_1^{1-2\sg} \nabla_\theta V)=0~~~~~~~~~& \textmd{in} ~(-\infty, 0) \times\mathbb{S}^n_+ ,\\
-\lim_{\theta_1\rightarrow 0^+} \theta_1^{1-2\sg} \partial_{\theta_1} V = V^p~~~~~~~~& \textmd{on} ~(-\infty, 0) \times \partial \mathbb{S}^n_+,\\
\end{cases}
\end{equation}
where 
$$
J_1=\frac{4\sg}{p-1} -(N-2\sg),~~~~~
J_2=\frac{2\sg}{p-1}\left(n-2\sg-\frac{2\sg}{p-1}\right).
$$
Multiplying \eqref{P30111} by $V_s$ and integrating, we have
\begin{equation}\label{P30122}
\aligned
& \int_{\mathbb{S}^n_+}\theta_1^{1-2\sg}V_{ss}V_s   - J_2\int_{\mathbb{S}^n_+}\theta_1^{1-2\sg}VV_s -\int_{\mathbb{S}^n_+}\theta_1^{1-2\sg}\nabla_\theta V \cdot \nabla_\theta V_s + \int_{\partial \mathbb{S}^n_+}V^pV_s\\
& = J_1 \int_{\mathbb{S}^n_+}\theta_1^{1-2\sg}(V_s)^2.
\endaligned
\end{equation}
For any $s\in (-\infty, 0)$, we define 
$$
\aligned
\widetilde{E}(s):=   &   \frac{1}{2}\int_{\mathbb{S}^n_+}\theta_1^{1-2\sg} (V_s)^2 -\frac{J_2}{2}\int_{\mathbb{S}^n_+}\theta_1^{1-2\sg} V^2 - \frac{1}{2}\int_{\mathbb{S}^n_+}\theta_1^{1-2\sg} |\nabla_\theta V|^2 \\
&   +\frac{1}{p+1}\int_{\partial \mathbb{S}^n_+} V^{p+1}.
\endaligned
$$
Then, by \eqref{P30122}, we get
\begin{equation}\label{P30133}
\frac{d}{ds} \widetilde{E}(s)= J_1 \int_{\mathbb{S}^n_+}\theta_1^{1-2\sg}(V_s)^2\geq 0.
\end{equation}
Here we have used the fact $J_1 > 0$ because $1< p < \frac{n+2\sg}{n-2\sg}$.  Hence, $\widetilde{E}(s)$ is non-decreasing in $s\in (-\infty, 0)$.

Now, rescaling back,  we have
$$
\aligned
& \int_{\mathbb{S}^n_+}\theta_1^{1-2\sg} (V_s)^2  \\
&  = \int_{\mathbb{S}^n_+}\theta_1^{1-2\sg} \left(\frac{2\sg}{p-1} r^{\frac{2\sg}{p-1}-1}U + r^{\frac{2\sg}{p-1}} U_r\right)^2r^2\\
& = r^{2\frac{(p+1)\sg}{p-1} -n}\int_{\partial^+\mathcal{B}_r^+} t^{1-2\sg}\left(\frac{4\sg^2}{(p-1)^2}  r^{-1}U^2 +\frac{4\sg}{p-1} U \frac{\partial U}{\partial \nu} +r |\frac{\partial U}{\partial \nu}|^2\right),
\endaligned
$$
$$
\int_{\mathbb{S}^n_+}\theta_1^{1-2\sg} |\nabla_\theta V|^2=r^{2\frac{(p+1)\sg}{p-1} -n+1}\int_{\partial^+\mathcal{B}_r^+} t^{1-2\sg} \left( |\nabla U|^2 - |\frac{\partial U}{\partial \nu}|^2 \right),
$$
$$
\int_{\mathbb{S}^n_+}\theta_1^{1-2\sg} V^2=r^{2\frac{(p+1)\sg}{p-1} -n-1}\int_{\partial^+\mathcal{B}_r^+} t^{1-2\sg} U^2,
$$
$$
\int_{\partial \mathbb{S}^n_+} V^{p+1} = r^{2\frac{(p+1)\sg}{p-1} -n+1}\int_{\partial B_r} u^{p+1}.
$$
Substituting these into \eqref{P30133} and noting  that $s=\ln r$ is non-decreasing in $r$, we easily obtain that $E(r;U)$ is also non-decreasing in $r \in (0, 1)$.
\ep

By the monotonicity of $E(r; U)$ we prove the following proposition, which will play an essential  role in deriving the lower bound in \eqref{Be01}.

\begin{proposition}\label{P303}
Let $U$ be a nonnegative solution of \eqref{Iso3} with $ \frac{n}{n-2\sg} < p < \frac{n+2\sg}{n-2\sg}$.  If
$$
\liminf_{|x| \rightarrow 0}|x|^{\frac{2\sg}{p-1}} u(x)=0,
$$
then
$$
\lim_{|x| \rightarrow 0}|x|^{\frac{2\sg}{p-1}} u(x)=0. 
$$
\end{proposition}
\bp
Suppose by contradiction that
$$
\liminf_{|x| \rightarrow 0}|x|^{\frac{2\sg}{p-1}} u(x)=0 ~~~~ \textmd{and} ~~~~  \limsup_{|x| \rightarrow 0}|x|^{\frac{2\sg}{p-1}} u(x)=C > 0. 
$$
Therefore, there exist two sequences of points $\{x_i\}$ and $\{y_i\}$ satisfying
$$
x_i \rightarrow 0, ~~~y_i \rightarrow 0~~~~\textmd{as}~ i \rightarrow \infty,              
$$
such that
$$
|x_i|^{\frac{2\sg}{p-1}} u(x_i)\rightarrow 0~~~ \textmd{and}~~~|y_i|^{\frac{2\sg}{p-1}} u(y_i)\rightarrow C>0~~~\textmd{as}~ i \rightarrow \infty.
$$
Let $g(r)=r^{\frac{2\sg}{p-1}} \bar{u}(r)$, where $\bar{u}(r)=\frac{1}{|\partial B_r|}\int_{\partial B_r} u$ denotes the spherical average of $u$ over $\partial B_r$. Then, by the Harnack inequality \eqref{P-3013}, we have
$$
\liminf_{r \rightarrow 0}g(r)=0 ~~~~ \textmd{and} ~~~~  \limsup_{r \rightarrow 0}g(r)=C > 0. 
$$
Hence, there exists a  sequence of local minimum points $r_i$ of $g(r)$ with
$$
\lim_{i\rightarrow \infty} r_i=0 ~~~~ \textmd{and} ~~~~ \lim_{i\rightarrow \infty} g(r_i)=0. 
$$
Let
$$
V_i(X)=\frac{U(r_i X)}{U(r_i e_1)},
$$
where $e_1=(1,0,\cdots,0)$.  It follows  from the Harnack inequality \eqref{P-3012} that $V_i$ is locally uniformly bounded away from the origin and satisfies 
\begin{equation}\label{P303001}
\begin{cases}
-\textmd{div}(t^{1-2\sg}  \nabla V_i)=0~~~~~~~~~& \textmd{in} ~\mathbb{R}^{n+1}_+,\\
\frac{\partial V_i}{\partial \nu^\sg}(x, 0)=\left(r_i^{\frac{2\sg}{p-1}} U(r_ie_1) \right)^{p-1} V_i^p(x, 0)~~~~~~~~& \textmd{on} ~\mathbb{R}^n \backslash \{0\}.
\end{cases}
\end{equation}
Note that by the Harnack inequality \eqref{P-3013}, $r_i^{\frac{2\sg}{p-1}} U(r_ie_1)\rightarrow 0$ as $i\rightarrow \infty$.  Then by Corollary 2.10 and Theorem 2.15  in \cite{J-L-X} that  there exists $\alpha >0$ such that for every $R > 1 > r > 0$
$$
\|V_i\|_{W^{1,2}(t^{1-2\sg}, \mathcal{B}_R^+ \backslash \overline{\mathcal{B}}_r^+)} + \|V_i\|_{C^\alpha (\mathcal{B}_R^+ \backslash \overline{\mathcal{B}}_r^+)} + \|v_i\|_{C^{2,\alpha}(B_R \backslash B_r)} \leq C(R, r),
$$
where $v_i(x)=V_i(x,0)$ and $C(R, r)$ is independent of $i$. Then after passing to a subsequence, $\{V_i\}$ converges to a nonnegative function $V  \in W_{loc}^{1,2}(t^{1-2\sg}, \overline{\mathbb{R}^{n+1}_+} \backslash \{0\}) \cap C^\alpha_{loc}(\overline{\mathbb{R}^{n+1}_+} \backslash \{0\})$ satisfying 
\begin{equation}\label{P303002}
\begin{cases}
-\textmd{div}(t^{1-2\sg}  \nabla V)=0~~~~~~~~~& \textmd{in} ~\mathbb{R}^{n+1}_+,\\
\frac{\partial V}{\partial \nu^\sg}(x, 0)=0~~~~~~~~& \textmd{on} ~\mathbb{R}^n \backslash \{0\}.
\end{cases}
\end{equation}
By a B\^{o}cher type theorem in \cite{J-L-X}, we have
$$
V(X)=\frac{a}{|X|^{n-2\sg}} + b,
$$
where $a, b$ are nonnegative constants.  Recall that $r_i$ are local minimum  of $g(r)$ for every $i$ and note that
$$
r^{\frac{2\sg}{p-1}} \bar{v}_i(r) =r^{\frac{2\sg}{p-1}}\frac{1}{|\partial B_r|} \int_{\partial B_r}v_i=\frac{1}{U(r_ie_1)}r^{\frac{2\sg}{p-1}}\bar{u}(r_i r)=\frac{1}{U(r_ie_1)r_i^{\frac{2\sg}{p-1}}} g(r_i r). 
$$
Hence, for every $i$, we have
\begin{equation}\label{P303003}
\frac{d}{dr} \left[ r^{\frac{2\sg}{p-1}} \bar{v}_i(r)\right] \Bigg|_{r=1} = \frac{r_i}{U(r_ie_1)r_i^{\frac{2\sg}{p-1}}} g^\prime (r_i )=0.
\end{equation}
Let $v(x)=V(x,0)$. Then we know that  $v_i(x) \rightarrow v(x)$ in $C_{loc}^2(\mathbb{R}^n \backslash \{0\})$. By \eqref{P303003}, we obtain 
$$
\frac{d}{dr} \left[ r^{\frac{2\sg}{p-1}} \bar{v}(r)\right] \Bigg|_{r=1} =0,
$$
which implies that
\begin{equation}\label{P303004}
a\left(\frac{2\sg}{p-1} - (n-2\sg)\right) +  \frac{2\sg b}{p-1}=0.
\end{equation}
On the other hand, by  $V(e_1)=1$, we have
\begin{equation}\label{P303005}
a+b=1.
\end{equation}
Combine \eqref{P303004} with \eqref{P303005}, we get
$$
a=\frac{2\sg}{(p-1)(n-2\sg)}~~~~~ \textmd{and} ~~~~~ b=1- \frac{2\sg}{(p-1)(n-2\sg)}.
$$
Since $\frac{n}{n-2\sg} < p$, we have $0 < a, b < 1$. Now we compute $E(r; U)$.

It follows from Proposition 2.19 in \cite{J-L-X} that $|\nabla_x V_i|$ and $|t^{1-2\sg} \partial_t V_i|$ are locally uniformly bounded in $C_{loc}^\beta (\overline{\mathbb{R}^{n+1}_+}\backslash \{0\})$ for some $\beta >0$. Hence,  there exists a  constant $C>0$ such that 
$$
|\nabla_x U(X)| \leq C r_i^{-1} U(r_i e_1)= o(1) r_i^{-\frac{2\sg}{p-1} - 1}~~~~~~ \textmd{for}  ~ \textmd{all} ~|X|=r_i
$$
and
$$
|t^{1-2\sg }\partial_t U(X)| \leq C r_i^{-2\sg} U(r_i e_1)= o(1) r_i^{-\frac{2\sg}{p-1} - 2\sg}~~~~~~ \textmd{for}  ~ \textmd{all} ~|X|=r_i. 
$$
Thus, by a direct computation, we can get 
$$
\lim_{i \rightarrow \infty} E(r_i; U)=0.
$$
By the monotonicity of $E(r; U)$, we obtain
$$
E(r; U) \geq 0~~~~~ \textmd{for} ~ \textmd{all} ~ r\in (0, 1).
$$
On the other hand, by the scaling invariance of $E(r; U)$, for every $i$, we have
$$
0 \leq E(r_i; U)=E\left(1; r_i^{\frac{2\sg}{p-1}} U(r_i X) \right) = E\left(1; r_i^{\frac{2\sg}{p-1}} U(r_i e_1)V_i \right).
$$
Hence, for every $i$, we have
$$
\aligned
 0 \leq  & \int_{\partial^+\mathcal{B}_1^+} t^{1-2\sg} | \frac{\partial V_i}{\partial \nu} |^2 + \frac{2\sg}{p-1} \int_{\partial^+\mathcal{B}_1^+} t^{1-2\sg}\frac{\partial V_i}{\partial \nu} V_i \\
& + \frac{1}{2}\frac{2\sg}{p-1}\left(\frac{4\sg}{p-1} - (n-2\sg)\right)   \int_{\partial^+\mathcal{B}_1^+} t^{1-2\sg} V_i^2  \\
& -  \frac{1}{2} \int_{\partial^+\mathcal{B}_1^+} t^{1-2\sg} |\nabla V_i|^2  +  \frac{1}{p+1}\int_{\partial B_1} \left(r_i^{\frac{2\sg}{p-1}} U(r_i e_1)\right)^{p-1}V_i^{p+1}.
\endaligned
$$
Letting $i \rightarrow \infty$, we obtain
$$
\aligned
 0  & \leq  \int_{\partial^+\mathcal{B}_1^+} t^{1-2\sg} | \frac{\partial V}{\partial \nu} |^2 + \frac{2\sg}{p-1} \int_{\partial^+\mathcal{B}_1^+} t^{1-2\sg}\frac{\partial V}{\partial \nu} V \\
& ~~~~ + \frac{1}{2}\frac{2\sg}{p-1}\left(\frac{4\sg}{p-1} - (n-2\sg)\right)   \int_{\partial^+\mathcal{B}_1^+} t^{1-2\sg} V^2 -  \frac{1}{2} \int_{\partial^+\mathcal{B}_1^+} t^{1-2\sg} |\nabla V|^2 \\
& =a^2(n-2\sg)^2 \int_{\partial^+\mathcal{B}_1^+} t^{1-2\sg} - a(n-2\sg)\frac{2\sg}{p-1}\int_{\partial^+\mathcal{B}_1^+} t^{1-2\sg} \\
& ~~~~ + \frac{1}{2}\frac{2\sg}{p-1}\left(\frac{4\sg}{p-1} - (n-2\sg)\right)   \int_{\partial^+\mathcal{B}_1^+} t^{1-2\sg}-\frac{1}{2}a^2(n-2\sg)^2\int_{\partial^+\mathcal{B}_1^+} t^{1-2\sg}\\
& = \frac{\sg}{p-1} \left(\frac{2\sg}{p-1} - (n-2\sg)\right)   \int_{\partial^+\mathcal{B}_1^+} t^{1-2\sg} <0.
\endaligned
$$
Note that in the last inequality we have used the fact  $\frac{2\sg}{p-1} - (n-2\sg) <0$ because $\frac{n}{n-2\sg} < p$.  Obviously, we get a contradiction. 
\ep

To characterize the "order" of an isolated singularity we establish the following sufficient condition for removability of isolated singularities.  For its proof,  we adapt the arguments from \cite{G-S}, but there are extra difficulties. Such as, we need extra efforts to derive the estimates of $U$ from its trace $u$. 
\begin{proposition}\label{P304}
Let $U$ be a nonnegative solution of \eqref{Iso3} with $ \frac{n}{n-2\sg} < p < \frac{n+2\sg}{n-2\sg}$.  If
\begin{equation}\label{P304A}
\int_{\epsilon \leq |x| \leq 1} u^{\frac{(p-1)n}{2\sg}} \leq c < +\infty
\end{equation}
with $c$ independent of $\epsilon$, then the singularity at the origin is removable, i.e., u(x) can be extended to a continuous solution in the entire ball $B_1$.
\end{proposition}
\bp
Let 
$$p_0=\frac{n-2\sg}{2\sg}\left(p-\frac{n}{n-2\sg}\right) ~~~~ \textmd{and} ~~~~ q_0=\frac{1}{2}(p_0+1)=\frac{1}{2}\frac{n-2\sg}{2\sg}(p-1).$$
Following Serrin \cite{S1}, we define, for $q\geq q_0$, $l>0$
$$
F(u)=
\begin{cases}
u^q~~~~~~~~~ & \textmd{for} ~ 0< u \leq l,\\
\frac{1}{q_0} \left[ ql^{q-q_0} u^{q_0} + (q_0- q) l^q \right]~~~ & \textmd{for} ~ l \leq u,
\end{cases}
$$
and
$$
G(u)=F(u)F^\prime(u) - q.
$$
Clearly, $F$ is a $C^1$ function of $u$ and $G$ is a piecewise smooth function of $u$ with a corner at $u=l$. Moreover, since $p_0 > 0$ (implied by $\frac{n}{n-2\sg} < p$), $F, G$ satisfy
\begin{equation}\label{P30411}
F \leq \frac{q}{q_0}l^{q-q_0} u^{q_0},~~~~~~~~ uF^\prime \leq qF,
\end{equation}
\begin{equation}\label{P30422}
|G|\leq FF^\prime,
\end{equation}
\begin{equation}\label{P30433}
G^\prime \geq 
\begin{cases}
\frac{1}{q}  p {F^\prime}^2~~~~~ & \textmd{for} ~ 0< u < l,\\
\frac{1}{q_0}  p_0 {F^\prime}^2~~~~~ & \textmd{for} ~ l \leq u.
\end{cases}
\end{equation}
For any $ 0< R <1$, let $\eta$ and $\bar{\eta}$ be nonnegative $C^\infty$ function with $0 \leq \eta, \bar{\eta} \leq 1$ in $\mathcal{B}_R=\{X\in \mathbb{R}^{n+1} : |X| <  R\}$,  $\eta$ having compact support in $\mathcal{B}_R$, and $\bar{\eta}$ vanishing in some neighborhood of the origin.  Take $(\eta\bar{\eta})^2 G(U)$ as a test function into \eqref{Solu}, we have
\begin{equation}\label{P30444}
\aligned
& \int_{\mathcal{B}_R^+} t^{1-2\sg} (\eta\bar{\eta})^2 G^\prime (U) |\nabla U|^2 + 2\int_{\mathcal{B}_R^+} t^{1-2\sg}\eta\bar{\eta} G(U)\nabla U \cdot \nabla (\eta\bar{\eta}) \\
& = \int_{B_R} (\eta\bar{\eta})^2 G(u)  u^p.
\endaligned
\end{equation}
Using \eqref{P30411} -- \eqref{P30433} and simplifying we obtain from \eqref{P30444}
\begin{equation}\label{P30455}
\int_{\mathcal{B}_R^+} t^{1-2\sg} (\eta\bar{\eta})^2 |\nabla (F(U))|^2\leq C(q) \bigg\{\int_{\mathcal{B}_R^+} t^{1-2\sg}|\nabla (\eta\bar{\eta})|^2 F^2 + \int_{B_R} (\eta\bar{\eta})^2 u^{p-1} F^2 \bigg\}.
\end{equation}
By H\"{o}lder and Proposition \ref{P3}, we have
$$
\aligned
\int_{B_R} (\eta\bar{\eta})^2 u^{p-1} F^2 & \leq \left( \int_{B_R} (u^{p-1})^{\frac{n}{2\sg}} \right)^{\frac{2\sg}{n}} \left( \int_{B_R} (\eta\bar{\eta} F)^{\frac{2n}{n-2\sg}} \right)^{\frac{n-2\sg}{n}} \\
& \leq C_{n, \sg}^2 \left( \int_{B_R} (u^{p-1})^{\frac{n}{2\sg}} \right)^{\frac{2\sg}{n}}   \left(\int_{\mathcal{B}_R^+} t^{1-2\sg} |\nabla (\eta\bar{\eta} F)|^2 \right) \\
& \leq C_{n, \sg}^2 \|u^{p-1}\|_{L^{\frac{n}{2\sg}}(B_R)} \bigg\{ \int_{\mathcal{B}_R^+} t^{1-2\sg} (\eta\bar{\eta})^2 |\nabla F|^2 \\
&  ~~~~  + \int_{\mathcal{B}_R^+} t^{1-2\sg} |\nabla(\eta\bar{\eta})|^2  F^2 \bigg\}.
\endaligned
$$
By the assumption \eqref{P304A}, we can choose $R$ small enough (depending on $q, n$ and $\sg$ ) such  that 
$$
\|u^{p-1}\|_{L^{\frac{n}{2\sg}}(B_R)} \leq \frac{1}{2} \frac{1}{C(q)C_{n, \sg}^2},
$$
Hence, from \eqref{P30455}, we obtain 
\begin{equation}\label{P30466}
\int_{\mathcal{B}_R^+} t^{1-2\sg} (\eta\bar{\eta})^2 |\nabla (F(U))|^2\leq C(q) \int_{\mathcal{B}_R^+} t^{1-2\sg}|\nabla (\eta\bar{\eta})|^2 F^2
\end{equation}
with a new constant $C(q)$.  Therefore, we have from \eqref{P30466}
\begin{equation}\label{P30477}
\int_{\mathcal{B}_R^+} t^{1-2\sg} (\eta\bar{\eta})^2 |\nabla (F(U))|^2\leq C(q) \bigg\{ \int_{\mathcal{B}_R^+} t^{1-2\sg}|\nabla \eta|^2 F^2 + \int_{\mathcal{B}_R^+} t^{1-2\sg}|\nabla \bar{\eta}|^2 F^2 \bigg\} 
\end{equation}
and
\begin{equation}\label{P30488}
\left( \int_{B_R} (\eta\bar{\eta} F)^{\frac{2n}{n-2\sg}} \right)^{\frac{n-2\sg}{n}}\leq C(q) \bigg\{ \int_{\mathcal{B}_R^+} t^{1-2\sg}|\nabla \eta|^2 F^2 + \int_{\mathcal{B}_R^+} t^{1-2\sg}|\nabla \bar{\eta}|^2 F^2 \bigg\}.
\end{equation}
For any $\epsilon > 0$ small enough, we choose $\bar{\eta}_\epsilon$ satisfy
\begin{equation}\label{Cut}
\bar{\eta}_\epsilon(X)=
\begin{cases}
0 ~~~~~ \textmd{for} ~   |X| \leq \epsilon, \\
1 ~~~~~ \textmd{for} ~   2\epsilon \leq |X| < R, 
\end{cases}
\end{equation}
and $|\nabla \bar{\eta}_\epsilon(X)| \leq \frac{c}{\epsilon}$ for all $X \in \mathcal{B}_R$.  By H\"{o}lder inequality
\begin{equation}\label{P30499}
\aligned
\int_{\mathcal{B}_R^+} t^{1-2\sg}|\nabla \bar{\eta}_\epsilon|^2 F^2 &  \leq \left( \int_{\mathcal{B}_{2\epsilon}^+} t^{1-2\sg}|\nabla \bar{\eta}_\epsilon|^{n+2-2\sg} \right)^{\frac{2}{n+2-2\sg}}  \left( \int_{\mathcal{B}_{2\epsilon}^+} t^{1-2\sg} F^{2\frac{n+2-2\sg}{n-2\sg}}\right)^{\frac{n-2\sg}{n+2-2\sg}} \\
& \leq  \frac{C}{\epsilon^{2}} \left( \int_{\mathcal{B}_{2\epsilon}^+} t^{1-2\sg}\right)^{\frac{2}{n+2-2\sg}}  \left( \int_{\mathcal{B}_{2\epsilon}^+} t^{1-2\sg} F^{2\frac{n+2-2\sg}{n-2\sg}}\right)^{\frac{n-2\sg}{n+2-2\sg}} \\
& \leq C \left( \int_{\mathcal{B}_{2\epsilon}^+} t^{1-2\sg} F^{2\frac{n+2-2\sg}{n-2\sg}}\right)^{\frac{n-2\sg}{n+2-2\sg}},
\endaligned
\end{equation}
where $C$ is a positive constant independent of $\epsilon$.  Since, by \eqref{P30411}, the Harnack inequality \eqref{P-3013} and \eqref{P-3011}, we have
\begin{equation}\label{P304100}
\aligned
\int_{\mathcal{B}_{2\epsilon}^+} t^{1-2\sg} F^{2\frac{n+2-2\sg}{n-2\sg}} & \leq C(l, q) \int_{\mathcal{B}_{2\epsilon}^+} t^{1-2\sg} U^{2q_0\frac{n+2-2\sg}{n-2\sg}} \\
& \leq C(l,q) \int_0^{2\epsilon} \left( \sup_{\partial^+ \mathcal{B}_s^+} U\right)^{2q_0\frac{n+2-2\sg}{n-2\sg}} s^{n+1-2\sg} ds\\
& \leq C(l,q) \int_0^{2\epsilon} \left( \inf_{\partial^+ \mathcal{B}_s^+} U\right)^{2q_0\frac{n}{n-2\sg}} s^{-2q_0\frac{4\sg(1-\sg)}{(p-1)(n-2\sg)}}s^{n+1-2\sg} ds\\
& \leq C(l,q) \int_0^{2\epsilon} \left( \inf_{\partial B_s} u\right)^{2q_0\frac{n}{n-2\sg}} s^{n-1} ds\\
& \leq C(l,q) \int_0^{2\epsilon} \left( \frac{1}{|\partial B_s|}\int_{\partial B_s} u^{2q_0\frac{n}{n-2\sg}}\right) s^{n-1} ds\\
& \leq C(l,q) \int_{B_{2\epsilon}} u^{2q_0\frac{n}{n-2\sg}} = C(l,q) \int_{B_{2\epsilon}} u^{\frac{(p-1)n}{2\sg}}. 
\endaligned 
\end{equation}
Here we have used the fact $2q_0\frac{n}{n-2\sg}=\frac{(p-1)n}{2\sg} >1$ because $ p-1 > \frac{2\sg}{n-2\sg}$. Now, we have from \eqref{P304A}, \eqref{P30499} and \eqref{P304100}
$$
\int_{\mathcal{B}_R^+} t^{1-2\sg}|\nabla \bar{\eta}_\epsilon|^2 F^2 \rightarrow 0 ~~~~ \textmd{as} ~ \epsilon \rightarrow 0.
$$
This together with \eqref{P30477} and \eqref{P30488}, we obtain 
\begin{equation}\label{P304111}
\int_{\mathcal{B}_R^+} t^{1-2\sg} \eta^2 |\nabla (F(U))|^2\leq C(q) \int_{\mathcal{B}_R^+} t^{1-2\sg}|\nabla \eta|^2 F^2 
\end{equation}
and
\begin{equation}\label{P304122}
\left( \int_{B_R} (\eta F)^{\frac{2n}{n-2\sg}} \right)^{\frac{n-2\sg}{n}}\leq C(q) \int_{\mathcal{B}_R^+} t^{1-2\sg}|\nabla \eta|^2 F^2. 
\end{equation}
Let $l \rightarrow +\infty$, since $F(u)\rightarrow u^q$, we have 
\begin{equation}\label{P304133}
\int_{\mathcal{B}_R^+} t^{1-2\sg} \eta^2 |\nabla U^q|^2 \leq C(q) \int_{\mathcal{B}_R^+} t^{1-2\sg}|\nabla \eta|^2 U^{2q} 
\end{equation}
and
\begin{equation}\label{P304144}
\left( \int_{B_R} (\eta u^q)^{\frac{2n}{n-2\sg}} \right)^{\frac{n-2\sg}{n}} \leq C(q) \int_{\mathcal{B}_R^+} t^{1-2\sg}|\nabla \eta|^2 U^{2q}. 
\end{equation}
By a similar estimate as in \eqref{P304100}, we can obtain
\begin{equation}\label{P304155}
\aligned
\int_{\mathcal{B}_R^+} t^{1-2\sg}|\nabla \eta|^2 U^{2q}  & \leq C(q) \int_0^R \left(\inf_{\partial^+\mathcal{B}_s^+} U^{2q-2q_0\frac{2(1-\sg)}{n-2\sg}} \right)s^{n-1} ds \\
& \leq C(q) \left( \int_{B_R} u^{2q} \right)^{1-\frac{q_0}{q}\frac{2(1-\sg)}{n-2\sg}}
\endaligned
\end{equation}
Inequality \eqref{P304133}, \eqref{P304144} and \eqref{P304155} can be iterated a \textit{finite} number of times to show that
$$
U \in W^{1,2}(t^{1-2\sg}, \mathcal{B}_R^+) ~~~~ \textmd{and} ~~~~ u\in L^q(B_R)  ~~~~ \textmd{for}  ~  \textmd{all} ~ q>0.
$$
Furthermore, $U$ satisfies 
$$
\begin{cases}
-\textmd{div}(t^{1-2\sg}  \nabla U)=0~~~~~~~~~& \textmd{in} ~ \mathcal{B}_R^+,\\
\frac{\partial U}{\partial \nu^\sg}(x, 0)=U^p(x, 0)~~~~~~~~& \textmd{on} ~ \partial^0\mathcal{B}_R^+. 
\end{cases}
$$
Indeed, for $\epsilon >0$ small, let $\bar{\eta}_\epsilon$ be a smooth cut-off function as in \eqref{Cut}.   Let $\psi \in C_c^\infty (\mathcal{B}_R^+ \cup \partial^0\mathcal{B}_R^+)$. It follows from \eqref{Solu} that
\begin{equation}\label{P304166}
\int_{\mathcal{B}_R^+} t^{1-2\sg} \nabla U \nabla (\psi \bar{\eta}_\epsilon)=\int_{\partial^0\mathcal{B}_R^+} U^p \psi \bar{\eta}_\epsilon.
\end{equation}
Since
$$
\left| \int_{\mathcal{B}_R^+} t^{1-2\sg} \nabla U \nabla \bar{\eta}_\epsilon \psi \right| \leq C \epsilon^{\frac{n-2\sg}{2}} \int_{\mathcal{B}_R^+} t^{1-2\sg} |\nabla U|^2 \rightarrow 0 ~~~~ \textmd{as} ~ \epsilon \rightarrow 0, 
$$
by the dominated convergence theorem, let $\epsilon \rightarrow 0$ in \eqref{P304166}, we obtain
$$
\int_{\mathcal{B}_R^+} t^{1-2\sg} \nabla U \nabla \psi =\int_{\partial^0\mathcal{B}_R^+} U^p \psi.
$$
Since $u \in L^q(B_R)$ for some  $q>\frac{n}{2\sg}$,  it follows from Proposition 2.10  in \cite{J-L-X}  that $U$ is H\"{o}lder continuous in $\overline{\mathcal{B}_{R/2}^+}$. 
The proof of the proposition is completed.
\ep

\begin{corollary}\label{COR}
Let $U$ be a nonnegative solution of \eqref{Iso3} with $ \frac{n}{n-2\sg} < p < \frac{n+2\sg}{n-2\sg}$.  Then either the origin 0 is a removable singularity or $\lim_{|X| \rightarrow 0} U(x, t)= +\infty$. 
\end{corollary}
\bp By Proposition \ref{P304}, if the origin is not a removable singularity, then  there exists a sequence of points $\{x_j\}$ such that
$$
r_j=|x_j| \rightarrow 0 ~~~~ \textmd{and} ~~~~U(x_j, 0) \rightarrow +\infty ~~~ \textmd{as} ~ j\rightarrow \infty.
$$
By the Harnack inequality \eqref{P-3012}, we have
$$
\inf_{|X|=r_j}U(X) \geq C^{-1} U(x_j, 0).
$$
By the maximum principle, 
$$
U(X) \geq \inf_{|X|=r_j, r_{j+1}} U(X) \geq C^{-1} \min (U(x_j, 0), U(x_{j+1}, 0)) ~~~ \textmd{in} ~ r_{j+1} \leq |X| \leq r_j.
$$
Hence, we have  $U(X) \rightarrow +\infty$ as $|X|\rightarrow 0$. 
\ep

\noindent {\bf Proof of Theorem \ref{THM01}.}  By Proposition \ref{P-301},
$$
u(x) \leq c |x|^{-\frac{2\sg}{p-1}}.
$$
If \eqref{Be01} does not hold, then
$$
\liminf_{x\rightarrow 0} |x|^{\frac{2\sg}{p-1}} u(x)=0.
$$
It follows Proposition \ref{P303} that
\begin{equation}\label{Con}
\lim_{x\rightarrow 0} |x|^{\frac{2\sg}{p-1}} u(x)=0
\end{equation}
We will prove that  the origin is a removable singularity.  It suffices to establish \eqref{P304A} by Proposition \ref{P304}. 

Let $\tau=\frac{n-2\sg}{p-1}(p-\frac{n}{n-2\sg})$ and $0 < \delta < 1$. Define 
$$
\Phi=|X|^{-\tau} - \delta t^{2\sg} |X|^{-(\tau + 2\sg)}. 
$$
Then we can choose $\delta$ small (depending only on $n$ , $\sg$ and $p$) such that
\begin{equation}\label{THM-1}
\begin{cases}
-\textmd{div}(t^{1-2\sg}  \nabla \Phi) \geq 0~~~~~~~~~& \textmd{in} ~\mathbb{R}^{n+1}_+,\\
\frac{\partial \Phi}{\partial \nu^\sg}(x, 0)=2\delta\sigma |x|^{-2\sg} \phi(x)~~~~~~~~& \textmd{on} ~\mathbb{R}^n \backslash \{0\},
\end{cases}
\end{equation}
where $\phi (x)=\Phi(x,0)=|x|^{-\tau}$. Take $\xi_1(s)\in C_c^\infty(\mathbb{R})$ satisfying  $0\leq\xi_1\leq 1$ in $\mathbb{R}$ and
$$
\xi_1(s)=\begin{cases}
0~~~~~~&\textmd{if}~|s|\leq 1,\\
1~~~~~~&\textmd{if}~|s|\geq 2.
\end{cases}
$$
Take $\xi_2(s)\in C_c^\infty(\mathbb{R})$ satisfying  $0\leq\xi_2\leq 1$ in $\mathbb{R}$ and
$$
\xi_2(s)=\begin{cases}
1~~~~~~&\textmd{if}~|s|\leq \frac{1}{2},\\
0~~~~~~&\textmd{if}~|s|\geq \frac{3}{4}.
\end{cases}
$$
For any  $\epsilon >0$ small,  we choose   $\zeta(X) \in C_c^\infty (\epsilon < |X| < 1)$ as follow
$$
\zeta(X)=\begin{cases}
\xi_1(\frac{|X|}{\epsilon})~~~~~~&\textmd{if}~|X|\leq \frac{1}{2},\\
\xi_2(|X|)~~~~~~&\textmd{if}~|X|\geq \frac{1}{2}.
\end{cases}
$$
Using $\zeta\Phi$ as a test function in \eqref{Solu},  and the divergence theorem we obtain
\begin{equation}\label{THM-2}
\aligned
\int_{B_1} u\zeta \phi \left(\frac{1}{\phi}\frac{\partial \Phi}{\partial \nu^\sg} - u^{p-1}\right) =  & \int_{\mathcal{B}_1^+} U\bigg\{ \textmd{div}(t^{1-2\sg} \nabla \Phi) \zeta  + 2 t^{1-2\sg} \nabla \zeta\cdot \nabla\Phi \\
& + \textmd{div}(t^{1-2\sg} \nabla \zeta) \Phi \bigg\} - \int_{B_1} u \phi \frac{\partial \zeta}{\partial \nu^\sg}. 
\endaligned
\end{equation}
Since $\zeta$ is radial, we have 
\begin{equation}\label{THM-3}
\frac{\partial \zeta}{\partial \nu^\sg} (x, 0)=-\lim_{t \rightarrow 0^+} t^{1-2\sg} \partial_t \zeta=0 ~~~~~\textmd{in} ~~ B_1.
\end{equation}
By \eqref{THM-1}, we obtain  
\begin{equation}\label{THM-4}
\int_{\mathcal{B}_1^+} U \textmd{div}(t^{1-2\sg} \nabla \Phi) \zeta \leq 0.
\end{equation}
Using \eqref{P-3011}  and the Harnack inequality \eqref{P-3012}, we have 
\begin{equation}\label{THM-5}
\aligned
& \int_{\mathcal{B}_1^+}  U \bigg\{ 2t^{1-2\sg}  \nabla \zeta\cdot \nabla\Phi  + \textmd{div}(t^{1-2\sg} \nabla \zeta) \Phi \bigg\}  \\
& \leq C_1 +  \int_{\mathcal{B}_{2\epsilon}^+ \backslash \mathcal{B}_{\epsilon}^+}     U \bigg\{ 2t^{1-2\sg}  |\nabla \zeta| |\nabla\Phi| + |\textmd{div}(t^{1-2\sg} \nabla \zeta)| \Phi \bigg\}\\
& \leq C_1 +  C\epsilon^{-\frac{2\sg}{p-1}-1} \int_{\mathcal{B}_{2\epsilon}^+ \backslash \mathcal{B}_{\epsilon}^+} \left( t^{1-2\sg}|X|^{-(\tau+1)} + |X|^{-(\tau+2\sg)} \right)  \\
&~~~~~~~~~~ + C\epsilon^{-\frac{2\sg}{p-1}-\tau} \int_{\mathcal{B}_{2\epsilon}^+ \backslash \mathcal{B}_{\epsilon}^+} t^{1-2\sg} \left( |\Delta \zeta| + \epsilon^{-1} \left|  \xi_1^\prime(\frac{|X|}{\epsilon})  \right|    \frac{1}{|X|} \right) \\
& \leq C_1 + C \epsilon^{n-(\tau + 2\sg)- \frac{2\sg}{p-1}} \leq C < +\infty
\endaligned
\end{equation}
with $C$ independent of $\epsilon$. Hence, we obtain 
$$
\int_{B_1} u\zeta \phi \left(\frac{1}{\phi}\frac{\partial \Phi}{\partial \nu^\sg} - u^{p-1}\right)  \leq C < +\infty.
$$
By \eqref{Con}, we have $u^{p-1}=o(|x|^{-2\sg})$, and while 
$$
\frac{1}{\phi}\frac{\partial \Phi}{\partial \nu^\sg}=2\delta\sigma |x|^{-2\sg}~~~~~~ \textmd{in} ~ B_1 \backslash \{0\}. 
$$
Therefore
\begin{equation}\label{THM-6}
\int_{B_1}u\zeta |x|^{-(n-\frac{2\sg}{p-1})} = \int_{B_1}u\zeta |x|^{-(\tau + 2\sg)} \leq C < +\infty
\end{equation}
with $C$ independent of $\epsilon$.  Again by \eqref{P-3011}, we have
$$
\aligned
\int_{2\epsilon \leq |x| \leq 1/2} u^{\frac{(p-1)n}{2\sg}} &   \leq C \int_{2\epsilon \leq |x| \leq 1/2} u |x|^{-\frac{2\sg}{p-1}[\frac{(p-1)n}{2\sg}-1]} \\
&  \leq  C \int_{B_1}u\zeta |x|^{-(n-\frac{2\sg}{p-1})} \leq C < +\infty
\endaligned
$$
with $C$ independent of $\epsilon$. Thus, we establish \eqref{P304A} and complete the proof of Theorem \ref{THM01}. 
\hfill$\square$


\begin{thebibliography}{10}

\bibitem{A-W} W. Ao, H. Chan, A.  DelaTorre, M.A. Fontelos, M. Gonz\'alez, J. Wei,  
 On higher dimensional  singularities for the fractional  Yamabe problem: a non-local  Mazzeo-Pacard program, arXiv:1802.07973.

\bibitem{A-D-G-W} W. Ao, A. DelaTorre, M. Gonz\'alez, J. Wei, A gluing approach for the fractional Yamabe problem with prescribed
isolated singularities,  {\it Preprint,} arXiv:1609. 08903.


\bibitem{A-C-G-W} W. Ao, H. Chan, M. Gonz\'alez, J. Wei, Existence of positive weak solutions for fractional Lane-Emden equations
with prescribed singular set. {\it Preprint.}


\bibitem{A} P. Aviles,  Local behavior of solutions of some elliptic equations, {\it Comm. Math. Phys.,}  {\bf 108}  (1987) 177-192.


\bibitem{B-V}  M.-F. Bidaut-V\'{e}ron, L. V\'{e}ron,  Nonlinear elliptic equations on compact Riemannian manifolds and asymptotics of Emden equations, {\it Invent. Math.,} {\bf  106}  (3) (1991) 489-539.


\bibitem{Cab-S} X. Cabr\'e, Y. Sire,  Nonlinear equations for fractional Laplacians, I: Regularity, maximum principles, and Hamiltonian
estimates,  {\it Ann. Inst. H. Poincar\'e Anal. Non Lin\'eaire,}  {\bf  31}  (2014)  23-53.


\bibitem{C-G-S}  L. Caffarelli, B. Gidas, J. Spruck,  Asymptotic symmetry and local behavior of semilinear elliptic equations with critical Sobolev growth,  {\it Comm. Pure Appl. Math. ,}  {\bf 42}  (1989) 271-297.


\bibitem{C-J-S-X} L. Caffarelli, T. Jin, Y. Sire, J. Xiong,  Local analysis of solutions of fractional semi-linear elliptic equations with isolated singularities,  {\it Arch. Ration. Mech. Anal.,}  {\bf 213} (1) (2014)  245-268.


\bibitem{C-S}  L. Caffarelli, L. Silvestre,  An extension problem related to the fractional Laplacian,  {\it Comm. Partial Differential Equations,}  {\bf  32} (8) (2007) 1245-1260.


\bibitem{Ch-G} S.-Y. A. Chang, M. Gonz\'alez, Fractional Laplacian in conformal geometry, {\it  Adv. Math., } {\bf 226}  (2011) 1410-1432.


\bibitem{C-Lin} C.-C. Chen, C.-S. Lin, On the asymptotic symmetry of singular solutions of the scalar curvature equations,  {\it Math. Ann.,} {\bf 313}  (1999) 
229-245.


\bibitem{C-L-L} W. Chen, C. Li, Y. Li, A direct method of moving planes for the fractional Laplacian, {\it Adv. Math.,} {\bf 308} (2017) 404-437.



\bibitem{D-D-W} J. D\'{a}vila,  L. Dupaigne,  J. Wei,  On the fractional Lane-Emden equation, {\it  Trans. Amer. Math. Soc.,} {\bf 369} (2017) 6087-6104.


\bibitem{D-G} A. DelaTorre,  M.d.M. Gonz\'alez,  Isolated singularities for a semilinear equation for the fractional Laplacian arising
in conformal geometry,  {\it Preprint. } To appear in Rev. Mat. Iber.


\bibitem{D-d-G-W} A. DelaTorre, M. del Pino, M.d.M. Gonz\'alez, J. Wei,  Delaunay-type singular solutions for the fractional Yamabe
problem,  {\it Math.  Ann.,} {\bf 369}  (2017)  597-626.


\bibitem{D-P-V} E. Di Nezza, G. Palatucci,  E. Valdinoci. Hitchhiker's guide to the fractional Sobolev spaces, {\it Bull. Sci. Math.,} {\bf 136}  (2012) 521-573.


\bibitem{F-F} M. Fall, V. Felli,  Unique continuation property and local asymptotics of solutions to fractional elliptic equations, 
{\it Comm. Partial Differential Equations,}  {\bf  39} (2014) 354-397.


\bibitem{F} R. H. Fowler, Further studies of Emden's and similar differential equations,  {\it Q. J. Math., Oxf. Ser.} {\bf 2}  (1931) 259-288.


\bibitem{G-S} B. Gidas, J. Spruck,  Global and local behavior of positive solutions of nonlinear elliptic equations,
    {\it  Comm. Pure Appl. Math.,} {\bf 34} (1981) 525-598.
    
    
\bibitem{G-M-S} M.d.M. Gonz\'alez, R. Mazzeo, Y. Sire, Singular solutions of fractional order conformal Laplacians,  {\it J. Geom. Anal.,}
{\bf 22}  (2012) 845-863.


\bibitem{G-Q} M.d.M. Gonz\'alez, J. Qing, Fractional conformal Laplacians and fractional Yamabe problems,  {\it Anal. PDE,  } {\bf 6} (2013) 1535-1576.

    
    
\bibitem{J-d-S-X} T. Jin, O.S. de Queiroz, Y. Sire, J. Xiong,  On local behavior of singular positive solutions to nonlocal elliptic
equations,  {\it Calc. Var. Partial Differential Equations,} {\bf 56}  (1) (2017) Art. 9,  25 pp.
    
    
\bibitem{J-L-X} T. Jin, Y.Y. Li, J. Xiong,  On a fractional Nirenberg problem, part I: blow up analysis and compactness of solutions,  {\it J. Eur. Math. Soc.,} {\bf 16} (2014) 1111-1171.


\bibitem{K-M-W} S. Kim, M. Musso, J. Wei, Existence theorems of the fractional Yamabe problem, arXiv:1603.06617.


\bibitem{K-M-P-S} N. Korevaar, R. Mazzeo, F. Pacard, R. Schoen,  Refined asymptotics for constant scalar curvature metrics with
isolated singularities,  {\it Invent. Math.,}  {\bf 135}  (2) (1999) 233-272.


\bibitem{L} P.-L. Lions,  Isolated singularities in semilinear problems, {\it J. Differential Equations,}  {\bf 38}  (1980) 441-450.


\bibitem{P-Q-S} P. Pol\'{a}\v{c}ik, P. Quittner, Ph. Souplet,  Singularity and decay estimates in superlinear problems via
    Liouville-type theorems, {\it  Duke Math. J.,} {\bf  139}  (2007) 555-579.
        


\bibitem{S1} J, Serrin, Local behavior of solutions of quasi-linear equation, {\it Acta. Math.,} {\bf 111} (1964) 247-302.

 
\bibitem{S2} J, Serrin, Isolated singularities  of solutions of quasi-linear equation, {\it Acta. Math.,} {\bf 113} (1965) 219-240.

\end{thebibliography}
\end{document}